\documentclass[11pt]{amsart}
\usepackage[pdftex]{graphicx}
\usepackage{amsmath}
\usepackage{color}
\usepackage{amsthm}
\usepackage{multicol}

\usepackage{setspace} 

\usepackage{amssymb}

\theoremstyle{definition}

\newtheorem{defn}{Definition}[section]

\theoremstyle{theorem}
\newtheorem{thm}{Theorem}[section]

\newtheorem{prop}[thm]{Proposition}
\newtheorem{cor}[thm]{Corollary}

\usepackage{enumitem}

\usepackage[margin=1.3in]{geometry}

\usepackage{url}

\begin{document}

\title{Stabilization of Boij-S\"{o}derberg Decompositions of Ideal Powers}
\author{Sarah Mayes-Tang}
\address{Quest University Canada, 3200 University Blvd., Squamish BC, V8B 0N8}
\email{sarah.mayes-tang@questu.ca}


\maketitle

\begin{abstract}
Given an ideal $I$ we investigate the decompositions of Betti diagrams of the graded family of ideals $\{ I^k \}_k$ formed by taking powers of $I$.  We prove conjectures of Engstr\"{o}m from \cite{Eng13} and show that there is a stabilization in the Boij-S\"{o}derberg decompositions of $I^k$ for $k>>0$ when $I$ is a homogeneous ideal with generators in a single degree. In particular, the number of terms in the decompositions with positive coefficients remains constant for $k>>0$, the pure diagrams appearing in each decomposition have  the same shape, and the coefficients of these diagrams are given by polynomials in $k$.  We also show that a similar result holds for decompositions with arbitrary coefficients arising from other chains of pure diagrams.
\end{abstract}


\section{Introduction}

Eisenbud and Schreyer's proof of the Boij-S\"{o}derberg conjectures provided a new way to study the graded Betti numbers of modules, and presented the possibility of understanding arbitrary modules in terms of  modules with simple Betti numbers. In particular, they showed in \cite{ES09} that the Betti diagram of any graded Cohen-Maculay module may be written uniquely as a positive linear combination of modules with pure resolutions; such an expression is called a  \textit{ positive Boij-S\"{o}derberg decomposition}.  In \cite{BSncm} Boij and S\"{o}derberg extended this result to non-Cohen-Macaulay  graded modules and investigated additional arbitrary decompositions by removing the positivity condition.  Over the past several years, much work has been done to extend these results and to explicitly describe the decompositions in certain cases.  Despite this interest, the meaning of the coefficients and pure diagrams occurring in the decompositions remains somewhat mysterious.  

One way to better understand these decompositions is to look for patterns in the decompositions of related modules. In this paper, we build on Engst\"{o}m's work from  \cite{Eng13} and consider Boij-S\"oderberg decompositions of the Betti diagrams of powers of an ideal. Given an ideal $I$, we expect that the generating set, Betti numbers, and thus corresponding Boij-S\"{o}derberg decompositions of its powers $I^k$ will become increasingly complicated as we increase $k$.  Indeed, this is what usually happens.  However, we restrict our attention to ideals $I$ generated by homogeneous polynomials of the same degree, there is a stabilization in the Betti numbers (\cite{LV04}, \cite{Singla07}, \cite{WH14}).  Engstr\"{o}m conjectured that, when $I$ is generated by monomials of the same degree, there is a corresponding stabilization in the Boij-S\"{o}derberg decompositions (\cite{Eng13}).  The main goal of this paper is to prove a version of this conjecture which applies to ideals generated by any homogeneous polynomials of the same degree.  

In particular, when we have an ideal $I$ with generators of the same degree, Theorem \ref{thm:mainthm} says that the number of pure diagrams appearing in the positive Boij-S\"{o}derberg decomposition of $\beta(I^k)$ will stabilize as $k$ gets large.  Further,  after this stabilization occurs, the pure diagrams in the decomposition of each $\beta(I^k)$ will have nonzero entries in the same arrangement and the coefficients will be given by polynomials in $k$.  Theorem \ref{thm:mainthm2} extends this idea to decompositions of $\beta(I^k)$ arising from arbitrary chains of pure diagrams.  See Section \ref{sec:examples} for an illustration of these claims.

With this stabilization in hand, we may ask questions about the behaviour of the decompositions of $\beta(I^k)$ as $k$ approaches infinity.  For example, do the members of certain translated families of pure diagrams dominate?  If so, are the shapes of these translated families consistent across different choices of $I$?  Preliminary work on these questions has yielded some interesting patterns.  We also view this theorem as a first step to understanding the Boij-S\"{o}derberg decompositions of graded families of ideals in general.  Can we characterize the graded families where stabilization occurs?  Is there a correspondence between the diagrams and coefficients that appear in the decompositions of related families of ideals?

Section 2 of this paper provides background on the two main tools used in the proof of Theorems \ref{thm:mainthm} and \ref{thm:mainthm2}:  Boij-S\"{o}derberg theory and the Betti diagrams of powers of ideals.  Section 3 includes statements and proofs of Theorems  \ref{thm:mainthm} and \ref{thm:mainthm2} as well as examples illustrating the stabilization.

\section{Background}
Throughout, $S$ will be a polynomial ring in $n$ variables with the standard grading over a ring of characteristic zero.  The Betti numbers of an ideal $I$ of $S$ may be displayed in a Betti table $\beta(I)$, with $\beta_{i,i+j}(I)$ appearing in the $i$th column and the $j$th row of the table.

\subsection{Boij-S\"{o}derberg Theory}
\label{sec:BSBackground}
A resolution is said to be \textit{pure of type $\textbf{d} = (d_0, \dots, d_l)$} if it has the form
$$S(-d_0)^{\beta_{0,d_0}} \leftarrow S(-d_1)^{\beta_{0,d_1}} \leftarrow \cdots \leftarrow S(-d_l)^{\beta_{0,d_l}}$$
so that the corresponding Betti table of a pure resolution has exactly one nonzero entry $\beta_{i,d_i}$ in each column; such a table is called a \textit{pure diagram}. In \cite{BS08a}, Boij and S\"{o}derberg conjectured that given a strictly increasing degree sequence $\textbf{d}$ there exists a Cohen-Macaulay module with a pure resolution of type $\textbf{d}$; Eisenbud, Fl{\o}ystad, and Weyman proved this result in  \cite{EFW}. 

To describe how  the Betti table of an arbitrary module is related to pure diagrams, we identify canonical pure diagrams with degree sequence $\textbf{d}$.  There are at least three standard ways of doing so; two of them are recorded below.

\begin{defn}
Let $\textbf{d}$ be a strictly increasing degree sequence.    ${\pi}(\textbf{d})$ is the diagram with entries
\begin{equation*} {\pi}(\textbf{d})_{i,j} := \left\{
     \begin{array}{lr}
      \prod_{k\neq i} \dfrac{1}{|d_k-d_i|}  & : j=d_i\\
       0 & : j\neq d_i
       
     \end{array}
     \right.
   \end{equation*}
Let $\alpha$ be the smallest integer such that each product $\prod_{k\neq i} \dfrac{1}{|d_k-d_i|}$ is integer for all $i$ and $k$.  Then $\widetilde{\pi}(\textbf{d}) := \alpha {\pi}(\textbf{d})$ is the smallest multiple of ${\pi}(\textbf{d})$ with integer entries.
\end{defn}

The following definition gives a partial ordering on pure diagrams. 

\begin{defn}
We say that $\pi(d_0, d_1, \dots, d_s) \leq \pi(d'_0, d'_1, \dots, d'_t)$ if $s \geq t$ and $d_i \leq d'_i$ for all $i=0,1, \dots, t$.  Note that this is simply the term-wise order if we write each sequence $(d_0, d_1, \dots, d_r)$ as $(d_0,\dots, d_r, \infty, \infty, \dots)$.  A totally ordered subset of the partially ordered set of pure diagrams is called a \textit{chain}.
\end{defn}

In \cite{ES09}, Eisenbud and Schreyer prove that the Betti table of every graded Cohen-Macaulay $S$-module can be expressed uniquely as a linear combination of a chain of pure diagrams with \textit{positive} rational coefficients.  Boij and S\"{o}derberg extended this result to all graded modules in \cite{BSncm}.  The following algorithm produces such a decomposition.  

\vspace{0.1in}

\noindent \textbf{Decomposition Algorithm} \cite{ES09} \newline
\textit{Input:} A graded $S$-module $M$ with Betti table $\beta(M)$. \newline 
\textit{Output:}  A decomposition
$$\beta(M) = \eta_0 \pi(\textbf{d}_0) + \cdots + \eta_m \pi(\textbf{d}_m)$$
such that the diagrams $\pi(\textbf{d}_0), \dots, \pi(\textbf{d}_m)$ form a chain, and each $\eta_i$ is a positive rational number.

\begin{enumerate}
\item BEGIN: Let $L$ be the empty list, and set $\beta:=\beta(M)$.
\item For $i=0, \dots, \max\{ i: \beta_{i,j} \neq 0 \text{ for some } j\}$, let $d_i = \min\{ j : \beta_{i,j} \neq 0\}$ and compute $\pi(\textbf{d})$.  
\item Compute $\eta$, the largest rational number such that $\beta':=\beta- \eta \pi(\textbf{d})$ has non-negative entries.
\item Add $(\eta, \pi(\textbf{d}))$ to the list $L$.  If $\beta'=0$ then END.  Otherwise, set $\beta = \beta'$ and go back to step 2.
\end{enumerate}

\begin{defn}
The unique decomposition produced by the algorithm above will be called the \textit{positive Boij-S\"{o}derberg decomposition} of $\beta(I)$.
\end{defn}

We emphasize that there is a unique way to write the $\beta(I)$ as a linear combination of pure diagrams with \textit{positive} coefficients.  For every $\beta(I)$, the Decomposition Algorithm yields this unique expression.  When we remove positivity requirement, we obtain additional decompositions of $\beta(I)$.  

\begin{defn}
An expression of $\beta(I)$ as a (not necessarily positive) linear combination of pure diagrams forming a maximal chain is called a \textit{decomposition} of $\beta(I)$ with respect to the maximal chain.
\end{defn}

If the least degree generator of $I$ has degree $M$ and $N = \text{reg}(I)$, the nonzero entries of $\beta(I)$ are restricted to columns 0 through $n$ and rows $M$ through $N$.  Therefore, when considering a decomposition of $\beta(I)$ as a linear combination of pure diagrams, we may restrict our attention to pure diagrams that only have nonzero entries in these positions.  Following \cite{BSncm}, we define $B_{M,N}$ to be the subspace of generated by Betti diagrams with entries in these columns, and $B^s_{M,N}$ to be the subspace of $B_{M,N}$ of diagrams satisfying the first $s$ Herzog-K\"{u}hl equations.  Proposition 1 of \cite{BSncm} states that any maximal chain of pure diagrams of codimension at least $s$ in $B_{M,N}$ form a basis of $B_{M,N}^s$.  Further, any maximal chain of diagrams in $B^s_{M,N}$ yields a unique (not necessarily positive)  decomposition of $\beta(I)$.

The following result describes formulas that may be used to find the (not necessarily positive) coefficients of the decomposition with respect to any maximal chain.  For precise formulas see \cite{BSncm}.

\begin{prop}[Prop. 4 of \cite{BSncm}]
\label{prop:coeffs}
Let $I$ be an ideal in a polynomial ring $S = K[x_1, \dots, x_n]$, where $N$ be the degree of the least generator of $I$ and $M=\text{reg}(I)$.  Suppose that $\pi_0< \pi_1 = \pi(d_0, \dots, d_m) < \pi_2$ is part of a maximal chain in $B^s_{M,N}$.  If $\pi_0$ and $\pi_2$ differ from $\pi_1$ in the same column $c$, the coefficient of $\pi_1$ in the decomposition of $\beta(I)$ with respect to this maximal chain is given by 
\begin{eqnarray}
\beta_{c, d_c} \prod_{p\neq c} |d_c-d_p|.
\end{eqnarray}
Otherwise, it is given by an expression of the form
\begin{eqnarray}
\sum_{i=0}^n \sum_{d=M}^{d_i(\pi_0)} (-1)^i \alpha \prod_{q \in S} (d_q - d) \beta_{i,d}(I)
\end{eqnarray}
for a rational constant $\alpha$ and an index set $S \subset \{0, 1, \dots, m\}$.  The choice of $\alpha$ and $S$ depends on whether the placement of the nonzero entries of $\pi_1$ differ from those of $\pi_0$ and $\pi_2$ in a column or in codimension.  $\alpha$ is either equal to $1$ or to $d_{p} - d_{q}$ for some column indices $p$ and $q$.
\end{prop}

\subsection{Stabilization of Betti Numbers of $I^k$}
Throughout this section we consider ideals $I \subseteq S=K[x_1, \dots, x_n]$ equigenerated in degree $r$.  Let $\beta_{i,j}(I^k)$ denote $(i,j)$th graded Betti number of $I^k$.  Note that under our assumption, all generators of $I^k$ are of degree $rk$ so  all nonzero Betti numbers of $I^k$ are of the form $\beta_{i,j+kr}(I^k)$ for some positive integer $j$.  

Several recent papers have studied the behaviour of the graded Betti numbers $\beta_{i,j+kr}(I^k)$ for equigenerated ideals $I$ and large $k$.  They build off of the work of Kodiyalam in \cite{Kod93} which shows that for $k>>0$ the total Betti numbers $\beta_i(I^k) = \sum_{j} \beta_{i,j+kr}(I^k)$ are polynomials in $k$.
  
The following result states that the shape formed by the nonzero entries in the Betti tables $\beta(I^k)$ stabilize as $k$ gets large.  It appears to have been proven independently by Lavila-Vidal, Singla, and Whieldon,   

\begin{prop} [\cite{LV04}, \cite{Singla07}, \cite{WH14}]\label{prop:stabilization}
Let $I$ be an ideal of the a polynomial ring $S$ that is equigenerated in degree $r$.  Then there exists a $k_0$ such that for all $k>k_0$,
$$\beta_{i,j+kr}(I^k) \neq 0 \text{ if and only if } \beta_{i,j+k_0r}(I^{k_0}) \neq 0.$$
\end{prop}

\begin{cor}\label{cor:powerswindow}
Let $I$ an ideal of a polynomial ring $S$ equigenerated in degree $r$.  Then there are integers $k_0$ and $N$ such that for every $k>k_0$, any decomposition of $\beta(I^k)$ as a linear combination of pure diagrams consists of pure diagrams in the window $B_{rk, rk+N}$.
\end{cor}

The next result extends Proposition \ref{prop:stabilization} to tell us that the nonzero Betti numbers $\beta_{i,j+kr}(I^k)$ are polynomial in $k$ for $k>>0$.  

\begin{prop} [Cor. 2.2 of \cite{Singla07}, Prop. 6.3.6 of \cite{LV04}]\label{prop:polylstabilization}
Let $I$ be an ideal in a polynomial ring $S$ that is equigenerated in degree $r$.  Then there exists a $k_0$ such that for all $k> k_0$ and all $i$ and $j$, $\beta_{i, rk+j}(I^k)$ is a polynomial function in $k$.  The  polynomial functions corresponding to nonzero entries have positive leading coefficients and are of degree less than the length of $I$.  
\end{prop}

We will use the following terminology to describe decompositions of $\beta(I^k)$.  

\begin{defn}
We say that a \textit{translated family of pure diagrams} (indexed by $k$) is a sequence of pure diagrams whose nonzero entries have a fixed shape beginning in a row given  by a linear function $l(k)$.\footnote{Engstr\"{o}m calls these families  \textit{translations of pure diagrams.}}  Such a family may be of one of two forms,  depending on whether we consider resolutions of ideals $J$ or the corresponding quotient rings $S/J$:  

\begin{multicols}{2}

\begin{tabular}{r| l l l l l }
& 0 & 1 & 2 & $\cdots$\\
\hline
0 &  & \\
$\vdots$ & \\
$l(k)$ & \multicolumn{4}{c}{A fixed shape for}\\ 
&  \multicolumn{4}{c}{nonzero entries}
\end{tabular}

\begin{tabular}{r| l l l l l }
& 0 & 1 & 2 & $\cdots$\\
\hline
0 & 1 & \\
$\vdots$ & \\
$l(k)$ & &\multicolumn{4}{c}{A fixed shape for}\\ 
& & \multicolumn{4}{c}{nonzero entries}
\end{tabular}
\end{multicols}

A \textit{translated family of chains} indexed by $k$ is a collection of chains $C(k)= \{\pi_0(k)<\pi_1(k)< \cdots < \pi_l(k)\}$ where each $\pi_i(k)$ is a translated family of pure diagrams.
\end{defn}

\section{Stabilization of Decompositions}

Throughout this section $I \subseteq S=K[x_1, \dots, x_n]$ is an ideal equigenerated in degree $r$.  In Theorem \ref{thm:mainthm} we describe the positive Boij-S\"{o}derberg decompositions of $\beta(I^k)$ and in Theorem \ref{thm:mainthm2} we address arbitrary decompositions of $\beta(I^k)$.

\begin{thm}\label{thm:mainthm}
Let $I$ be a homogeneous ideal in a polynomial ring $S$ with all generators of degree $r$.  Then there are integers $m$ and $K$ such that for all $k> K$, the positive Boij-S\"{o}derberg decomposition of $\beta(I^k)$ is of the form
$$\beta(I^k) = w_1(k)\pi_1(k)+\cdots + w_m(k)\pi_m(k)$$
where each $w_i(k)$ is a polynomial in $k$ with rational coefficients and $\pi_1(k) < \pi_2(k) < \cdots < \pi_m(k)$ is a translated family of chains.  
\end{thm}

\begin{thm}\label{thm:mainthm2}
Let $I$ be a homogeneous ideal in a polynomial ring $S$ with all generators of degree $r$.  Suppose that $N$ is as in Corollary \ref{cor:powerswindow}.  Then there exists a $K$ such that for all $k>K$ and all translated families of maximal chains  $\pi(\textbf{d}^0(k))< \cdots < \pi(\textbf{d}^l(k))$ in $B_{rk,rk+N}$, the Boij-S\"{o}derberg decomposition of $\beta(I^k)$ with respect to the appropriate chain is of the form 
$$\beta(I^k) = w_0(k) \pi(\textbf{d}^0(k)) + \cdots + w_l(k) \pi(\textbf{d}^l(k))$$
where each $w_i(k)$ is a polynomial in $k$ with rational coefficients.

\end{thm}

\subsection{Example}
\label{sec:examples}

In this section we present an example illustrating the results of Theorems \ref{thm:mainthm} and \ref{thm:mainthm2}.  Note that this is a simpler version of the example in \cite{Eng13}.

Consider the ideal $I=(x_1x_2,x_2x_3,x_3x_4,x_4x_5) \subseteq K[x_1,\dots, x_5]$.  For $k \geq 3$ the ideals $I^k$ have Betti tables of the following shape.  

\begin{center}

\begin{tabular}{r| l l l l l }
& 0 & 1 & 2 & 3\\
\hline
$2k$ & $\ast$ & $\ast$ &$\ast$ &$\ast$ \\
$2k+1$ & - &$\ast$ &$\ast$ & -\\
\end{tabular}
\end{center}

Engstr\"{o}m and Nor\'{e}n (\cite{EN12}) showed that, for $k \geq 3$, 

\begin{eqnarray*}
\beta_{0,2k}(I^k) = \dfrac{1}{6} k^3+k^2+\dfrac{11}{6}k + 1 &&
\beta_{1,2k+1}(I^k) = \dfrac{1}{2} k^3+\dfrac{3}{2}k^2 + k\\
\beta_{1,2k+2}(I^k) = k &&
\beta_{2,2k+2}(I^k) = \dfrac{1}{2} k^3-\dfrac{1}{2}k \\
\beta_{2,2k+3}(I^k) = k && 
\beta_{3,2k+3}(I^k) = \dfrac{1}{6} k^3-\dfrac{1}{2}k^2+\dfrac{1}{3}k 
\end{eqnarray*}

We first follow the approach of Theorem \ref{thm:mainthm}.  The positive Boij-S\"{o}derberg decomposition resulting from the Decomposition Algorithm has 5 summands.  They involve pure diagrams from the following translated families of pure diagrams.

\begin{multicols}{5}
\begin{tabular}{r| l l l l  }
& 0 & 1 & 2 & 3\\
\hline
$2k$ & $\ast$ & $\ast$ &$\ast$ &$\ast$ \\
\end{tabular}

\begin{tabular}{r| l l l   }
& 0 & 1 & 2 \\
\hline
$2k$ & $\ast$ & $\ast$ &$\ast$  \\
\end{tabular}

\begin{tabular}{r| l l l   }
& 0 & 1 & 2 \\
\hline
$2k$ & $\ast$ & $\ast$ &-  \\
$2k+1$ & - &-&$\ast$\\
\end{tabular}
 
 \quad
 
\begin{tabular}{r| l l    }
& 0 & 1  \\
\hline
$2k$ & $\ast$ & - \\
$2k+1$ & - &$\ast$\\
\end{tabular}

\begin{tabular}{r| l     }
& 0   \\
\hline
$2k$ & $\ast$  \\
$2k+1$ & - \\
\end{tabular}

\end{multicols}

In particular, for $k\geq 3$
\begin{eqnarray*}
\beta(I^k) = &&(k^3-3k^2+2k) \enspace \pi(2k,2k+1,2k+2,2k+3) + (3k^2-3k) \enspace \pi(2k,2k+1,2k+2) \\
		&+& (6k) \enspace \pi(2k,2k+1,2k+3) + (2k) \enspace \pi(2k,2k+2) + \pi(2k).
\end{eqnarray*}

It is interesting to note that the stabilization does not occur immediately:  the positive Boij-S\"{o}derberg decompositions of $\beta(I)$ and $\beta(I^2)$ contain three and four summands, respectively.

Under the approach of Theorem \ref{thm:mainthm2}, the positive Boij-S\"{o}derberg decompositions of the ideals $I^k$  for $k \geq 3$ arise from the maximal chains
\begin{eqnarray*}  && \pi(2k,2k+1, 2k+2, 2k+3) \leq \pi(2k,2k+1,2k+2,2k+4) \leq \pi(2k,2k+1, 2k+2)  \\ && \leq \pi(2k,2k+1,2k+3) \leq \pi(2k,2k+2,2k+3) \leq \pi(2k,2k+2) \leq \pi(2k) \leq \pi(2k+1)
\end{eqnarray*}
in $B_{2k,2k+1}$.

We may also consider the decomposition of $\beta(I^k)$ with respect to the maximal chain 
\begin{eqnarray*}
&& \pi(2k,2k+1,2k+2,2k+3) \leq\pi(2k,2k+1,2k+2,2k+4) \leq \pi(2k,2k+1,2k+2)  \\ && \leq \pi(2k,2k+1,2k+3)  \leq \pi(2k,2k+2,2k+3) \leq \pi(2k+1,2k+2,2k+3)  \\ && \leq \pi(2k+1,2k+2) \leq \pi(2k+1)
\end{eqnarray*}
in $B_{2k,2k+1}$.  
For $k\geq 3$, the decomposition of $\beta(I^k)$ with respect to this chain is
\begin{eqnarray*}
\beta(I^k) &=& (k^3-3k^2+2k)\enspace \pi(2k,2k+1,2k+2,2k+3) + (3k^2-3k) \enspace \pi(2k,2k+1,2k+2)  \\ && + (6k)  \enspace \pi(2k,2k+1,2k+3) + (6k+6) \enspace \pi(2k,2k+2,2k+3) \\ && + (-4k-4)\enspace \pi(2k+1,2k+2,2k+3) +  (2k+1) \enspace \pi(2k+1,2k+2) + \pi(2k+1)
\end{eqnarray*}

Notice that all coefficients are given by polynomials in $k$ and that the coefficient of $\pi(2k+1,2k+2,2k+3)$ is negative.  Similar results hold for decompositions of $\beta(I^k)$ with respect to any other maximal chain in $B_{2k,2k+1}$.

\subsection{Stabilization of Positive Boij-S\"{o}derberg Decompositions}

\begin{proof}[Proof of Theorem \ref{thm:mainthm}]
Let $\beta(I^k)$ denote the Betti diagram of $I^k$.  Let $k_0$ be such that for all $i, j$ and all $k>k_0$ either $\beta_{i,j+rk}(I^k)=0$  or there are polynomials $P_{i,j}(k)$ that satisfy
$$\beta_{i,j+rk}(I^k) = P_{i,j}(k);$$
 such a $k_0$ is guaranteed by Propositions \ref{prop:stabilization} and \ref{prop:polylstabilization}.

We may use the Decomposition Algorithm given in Section \ref{sec:BSBackground} to determine the Boij-S\"{o}derberg decomposition of $\beta(I^k)$ for $k>>0$.  In particular, we will see that for $k>>0$ the procedure returns a list $L$ of pairs $(w_i(k), \pi_i(k))$ such that each $w_i(k)$ is a polynomial in $k$ with rational coefficients, the collection $\{ \pi_i(k) \}_k$ is a translated family of pure diagrams for all $i$, and $\beta(I^k) = \sum w_i(k) \pi_i(k)$.  

\begin{enumerate}
\item Let $L$ be the empty list.  Set $K=k_0$ and $\beta(k) = \beta(I^k)$ for all $k$.  Let $P_{i,j}(k) = \beta(k)_{i,j+rk}$ be the polynomials that determine the nonzero entries of the Betti table of $I^k$ for $k > K$.

\item   Let $\textbf{d}(k) = (d_0(k), d_1(k), \dots, d_{l}(k))$ be the degree sequence that corresponds to the uppermost nonzero entries in each column of $\beta(k)$; that is, set 
$$d_i(k) = \min_i \{ j : \beta(k)_{i, j} \neq 0 \}$$
for $i=0, 1, \dots, l:=\max\{ i : \beta(k)_{i,j} \neq 0 \text{ for some } j \}$.
The diagrams  $\beta(k)$ have the same shape for all $k> K$, so for all $i$ there is an integer $d_i$ such that 
$$d_i(k) = d_i +rk$$
for $k > K$.  Thus, the differences used to compute the entries in $\pi(\textbf{d}(k))$ are the same for all $k$: $|(d_i(k)-d_{i'}(k)| = |(d_i +rk)-(d_{i'}+rk)| = |d_i-d_{i'}|$.  Therefore,

$$\pi_{i,j}(\textbf{d}(k)) =  \left\{
     \begin{array}{lr}
         \prod_{i\neq i'} \dfrac{1}{|d_i-d_{i'}|}& :j=d_i+rk\\
       0 & : j\neq d_i+rk
       
     \end{array}
     \right.
   $$
for all $i=0, \dots, l_s$ and all $k>K$. 
Notice that  $\pi(\textbf{d}(k))$ is a translated family of pure diagrams.

\item  For each $k$, we need to find the greatest rational number $\eta(k)$ such that $\beta(k) - \eta(k)\pi(\textbf{d}(k))$ has nonzero entries.   By the calculations in step 2, this is equivalent to finding the greatest $\eta(k)$ such that
$$P_{i, d_i}(k) - \eta(k)  \prod_{i\neq i'} \dfrac{1}{|d_i-d_{i'}|} \geq 0$$
for all $i=0, \dots, l$.

Set $$Q_{i}(k) = P_{i,d_i}(k) \prod_{i\neq i'}|d_i-d_{i'}| ;$$ 
each $Q_i(k)$ is a polynomial function of $k$.  Then for each $i$ we need to find the greatest $\eta(k)$ such that 
$\eta(k) \leq Q_{i}(k)$
holds for all $i$.  This means that 
$\eta(k) = \min_i \{ Q_i(k) \}.$
Notice that the same choice of $Q_I(k)$ will work for all $k >>0$.  Specifically, we can set $\kappa>K$ to be a number larger than the greatest $k$-coordinate that occurs as a point of intersection of the polynomials $Q_i(k)$.  Then for $k >\kappa$, 
$$\eta(k) := Q_I(k)$$
and $\eta(k)$ is an polynomial in $k$.

\item Add $(\eta(k), \pi(\textbf{d}(k)))$ to the list $L$ and set $K=\kappa$.  Set $\beta'(k):=\beta(k) - \eta(k) \pi(\textbf{d}(k))$.  If $\beta'(k)=0$ for all $k >K$ then  END.  

Otherwise, note that

$$\beta'_{i,j}(k) =  \left\{
     \begin{array}{lr}
        P_{i, j}(k) - \eta(k)   \prod_{i\neq i'} \dfrac{1}{|d_i-d_{i'}|}& :j=d_i+rk\\
       \beta_{i,j}(I^k) & : j\neq d_i+rk       
     \end{array}
     \right.
   $$
  Therefore,  for $k>K$ each entry in $\beta'(k)$ is either a polynomial in $k$ for or zero.  Notice that given our  choice of $\eta(k)$, $\beta'(k)$ has at least one more zero entry than $\beta(k)$, and that all $\beta'(k)$ have the same shape when $k >K$.  This means that all assumptions used in steps 2 through 4 hold for $\beta'(k)$.  Set $\beta(k):=\beta'(k)$ and let $P_{i,j} = \beta_{i,j+rk}(k)$.  REPEAT from step 2.
  
\end{enumerate}
\end{proof}

\subsection{Stabilization of General Decompositions}

In this section, we give a proof of Theorem \ref{thm:mainthm2} and show how Theorem \ref{thm:mainthm} follows as a corollary.  We chose to include both proofs of Theorem \ref{thm:mainthm} in this paper, as  the above proof is more concrete and insightful, but Theorem \ref{thm:mainthm2} is a more general result.

\begin{proof}[Proof of Theorem \ref{thm:mainthm2}]

Let $k_0$ be such that for all $k >k_0$ either $\beta_{i, rk+j} =0$ or $\beta_{i, rk+j}(I^k) = P_{i,j}(k)$ is a polynomial function in $k$ for all $i,j$.  Note that for $k >k_0$ the pure diagrams occurring in any decomposition of $\beta(I^k)$ will be a subset of $B_{rk,rk+N}$.  

Any maximal chain in $B_{rk,rk+N}$ is a member of a translated family of maximal chains of the form $\pi(\textbf{d}^0(k)) < \pi(\textbf{d}^1(k)) < \cdots < \pi(\textbf{d}^l(k))$ where
$$\textbf{d}^i(k) = (d_0^i+rk, \dots, d_p^i+rk).$$
Choose one such family $\pi(\textbf{d}^0(k))< \cdots < \pi(\textbf{d}^l(k))$ indexed by $k$ and consider the expansion of  $\beta(I^k)$ with respect to the $k$th chain in the family:
$$\beta(I^k) = w_0(k) \pi(\textbf{d}^0(k)) + \cdots + w_l(k) \pi(\textbf{d}^l(k)).$$

The coefficients $w_j(k)$ are given by Proposition \ref{prop:coeffs}.  Notice that the pure diagrams in each sub-chain $\pi(\textbf{d}^{j-1}(k)) < \pi(\textbf{d}^{j}(k)) < \pi(\textbf{d}^{j+1}(k))$ have the same shape for all $k>>0$.  This means that each pair $(\pi(\textbf{d}^{j-1}(k)),  \pi(\textbf{d}^{j}(k)))$ and $(\pi(\textbf{d}^{j}(k)),  \pi(\textbf{d}^{j+1}(k)))$ will differ in either codimension or in the same column for all large values of $k$.  Therefore, the same formula will be used to find $w_j(k)$ for all $k>>0$.

If the pairs both differ in the same column $c$ then for $k>>0$, 
\begin{eqnarray*}
w_j(k)&=& \beta_{c,\textbf{d}^j(k)_c} (I^k) \prod_{p\neq c} |\textbf{d}^j(k)_c - \textbf{d}^j(k)_p| \\
&=& \beta_{c, d^j_c+rk} (I^k) \prod_{p\neq c} |(d_c^j+rk) - (d_p^j+rk)| \\
&=& P_{i,d_c^j}(k) \prod_{p\neq c} |d_c^j - d_p^j| \\
\end{eqnarray*}

Otherwise, there is a rational constant $\alpha$ and an index set $S$ such that for all $k>>0$, 
\begin{eqnarray*}
w_j(k) &=& \sum_{i=0}^n \sum_{d=rk}^{\pi(\textbf{d}^{j-1}(k))_i} (-1)^i \alpha \prod_{q \in S} (\textbf{d}^j(k)_q - d) \beta_{i,d}(I^k)\\
&=& \sum_{i=0}^n \sum_{d=rk}^{d_i^{j-1}+rk} (-1)^i \alpha \prod_{q \in S} (d_q^j + rk - d) \beta_{i,d}(I^k)\\
&=& \sum_{i=0}^n \sum_{d'=0}^{d_i^{j-1}} (-1)^i \alpha \prod_{q \in S} (d_q^j+rk - (d'+rk)) \beta_{i,d'+rk} (I^k)\\
&=& \sum_{i=0}^n \sum_{d'=0}^{d_i^{j-1}} (-1)^i \alpha \prod_{q \in S} (d_q^j-d') P_{i,d'}(k)\\
\end{eqnarray*}

In either case, $w_j(k)$ is a polynomial in $k$ with rational coefficients as claimed.  
\end{proof}

With this proof in hand, Theorem \ref{thm:mainthm} follows as a corollary.
\begin{proof}[Proof of Theorem \ref{thm:mainthm}, version 2]
Proposition 5 of \cite{BSncm} guarantees that exactly one chain in $B_{rk,rk+N}$ will yield entirely positive coefficients in the corresponding decomposition of $\beta(I^k)$ for each $k$.  We claim that for $k>>0$, all of these chains will belong to the same translated family of chains.

Indeed, the coefficients $w_j(k)$ for a given translated family of chains are polynomials by Theorem \ref{thm:mainthm2}.  Therefore, there exists a $k_1 \geq k_0$  such that for all $k \geq k_1$, each $w_j(k)$ is either positive, negative, or zero.  Combining this fact with the statement in the previous paragraph, exactly one translated family of chains yields coefficients that are all positive for all  large $k$. Therefore, the positive Boij-S\"{o}derberg decomposition of  $\beta(I^k)$ for each $k> k_1$ arises from the same translated family of chains.
\end{proof}

\section*{Acknowledgements}

I would like to thank Christine Berkesch Zamaere for helpful feedback on the results in this paper, and Daniel Erman for introducing Boij-S\"{o}derberg theory to me. Calculations in this paper were performed using the computer software Macaulay2 \cite{Macaulay2}.

\bibliography{BSBib}
\bibliographystyle{amsalpha}
\nocite{*}

\end{document}